\documentclass[a4paper,11pt]{article}
\usepackage{geometry}
\usepackage[page]{totalcount}
\title{Filling 1-cycles, 2-cycle complexity, and torsion growth}

\date{}
 
\usepackage[T1]{fontenc}
\usepackage{comment}
\usepackage{microtype}
\usepackage{pinlabel}
\usepackage{float}
\usepackage{subfig}
\usepackage{bbm}
\usepackage{amssymb, amsthm,thmtools,cite,mathtools,setspace,enumitem,tikz-cd}
\usepackage{hyperref}
\usepackage{cleveref}
\usepackage{thm-restate}
\usepackage{bm}
\usepackage[T1]{fontenc} 

 \usepackage{AlegreyaSans}
\usepackage[T1]{fontenc}
\usepackage[p,osf]{cochineal}
\usepackage[stix2]{newtxmath}
\usepackage[cal=boondoxo]{mathalfa}

\makeatletter
\newcommand\RedeclareMathOperator{%
	\@ifstar{\def\rmo@s{m}\rmo@redeclare}{\def\rmo@s{o}\rmo@redeclare}%
}
\newcommand\rmo@redeclare[2]{%
	\begingroup \escapechar\m@ne\xdef\@gtempa{{\string#1}}\endgroup
	\expandafter\@ifundefined\@gtempa
	{\@latex@error{\noexpand#1undefined}\@ehc}%
	\relax
	\expandafter\rmo@declmathop\rmo@s{#1}{#2}}
\newcommand\rmo@declmathop[3]{%
	\DeclareRobustCommand{#2}{\qopname\newmcodes@#1{#3}}%
}

\@onlypreamble\RedeclareMathOperator
\makeatother

\DeclareSymbolFont{bbold}{U}{bbold}{m}{n}

\newtheorem{thm}{Theorem}
\numberwithin{thm}{section}

\newtheorem{lem}[thm]{Lemma}

\newtheorem{prop}[thm]{Proposition}

\newtheorem{cor}{Corollary}

\RedeclareMathOperator{\fill}{\mathtt{fill}}

\DeclareMathOperator{\vol}{vol}

\DeclareMathOperator{\diam}{diam}

\DeclareMathOperator{\len}{{len}}

\newcommand{\Q}{{\mathbb Q}}
\newcommand{\Z}{{\mathbb Z}}

\newcommand{\C}{{\mathbb C}}
\newcommand{\R}{{\mathbb R}}

\newcommand{\cA}{\mathcal{A}}\newcommand{\cB}{\mathcal{B}}

\newcommand{\cT}{\mathcal{T}}

\newcommand{\e}{\varepsilon}
\newcommand{\G}{\Gamma}

\renewcommand{\S}{\Sigma}

\newcommand{\fc}{\mathfrak c}

\newcommand{\PSL}{{\mathsf{PSL}}}

\newcommand{\norm}[1]{\left\lVert#1\right\rVert}

\renewcommand{\d}{ \partial}

\author{Cameron Gates Rudd}
\definecolor{darkspringgreen}{rgb}{0.09, 0.45, 0.27}
\definecolor{darkmidnightblue}{rgb}{0.0, 0.2, 0.4}

\hypersetup{
    colorlinks=true,
    linkcolor=darkmidnightblue,
    citecolor=darkmidnightblue,
    filecolor=darkmidnightblue,
    urlcolor=darkmidnightblue
}

\newcommand{\rank}{{\mathrm{rank}~}}

\newcommand{\genus}{{\mathrm{genus}}}

\renewcommand{\fill}{{\mathsf{fill}}}

\newcommand{\ta}{{\underline{\mathtt{a}}}}
\newcommand{\tb}{{\underline{\mathtt{b}}}}

\newcommand{\ua}{{\underline{\alpha}}}
\newcommand{\ub}{{\underline{\beta}}}

\newcommand{\normalsub}{\unlhd}

\begin{document}

\maketitle

\begin{abstract}
    For 2-complexes with trivial first Betti number, we study quantitative connections between filling inequalities for 1-cycles, the complexity of the second homology, and the size of the torsion part of the first homology. For 2-complexes whose fundamental groups have property $(\tau)$, we prove a geometric lower bound on the logarithm of the size of the first homology of finite covers.
\end{abstract}

Let $\G$ be a finitely generated group with finite abelianization $\G^{ab}$; we call such groups rationally perfect. One can ask: "how big is $\#\G^{ab}?$"
If $X $ is a presentation 2-complex for a rationally perfect group $\G$, $\kappa(X)$ is the length of the longest relator in the corresponding presentation and $\rank(X)$ the number of generators in the presentation, then an oft used inequality attributed to Gabber says
\begin{align}
    \log\#\G^{ab} \leq \rank(X)\cdot\log\kappa(X).
\end{align}

This estimate is key in an ever increasing literature of torsion growth vanishing results; see for instance \cite{AbertGelanderNikolov,cheaprebuilding, Uschold, AGHK, Fruchter}.
Our principal aim in this note is to derive complementary inequalities bounding the size of $\G^{ab}$ from below using the geometry of the presentation complex $X$.

For a 2-complex whose fundamental group has property $(\tau)$ (with respect to \emph{all} finite index normal subgroups), we can prove the following lower bound complementing (1).

\begin{thm}\label{thm:intro tau}
    Let $X_0$ be a finite 2-complex whose fundamental group $\G_0$ has property $(\tau)$. Let $\G\normalsub \G_0$ be a finite index normal subgroup. Let $X =\widetilde X_0/\G$. Then there is a constant $C$ depending only on $X_0$ such that  $$
(\xi_2(X;\Z)-1)\rho_1(X;\Z)-C(1+\log([\G_0:\G])) \leq  C\log\#\G^{ab}.
$$ In particular, for a residual tower of covers $X_i\to X_0,$ we have $$\xi_2(X_i;\Z)\rho_1(X_i;\Z)-o([\G_0:\G_i])\leq C\log\#\G^{ab}_i.$$
\end{thm}

The main quantities appearing in Theorem \ref{thm:intro tau} are precisely defined later in this note, but briefly, $\xi_2(X;\Z)$ is the largest successive minimum for $H_2(X;\Z)$ using the $\ell^1$ norm, and $\rho_1(X;\Z)$ is an isoperimetric constant encoding how difficult it is to fill an integral $1$-cycle with an integral $2$-chain.

We point out that Gabber's inequality (1) combined with Theorem \ref{thm:intro tau} implies (in the present setting) that there is a uniform constant $C$ such that $$(\xi_2(X_i;\Z)-1)\rho_1(X_i;\Z)\leq Cd_i,$$ where $d_i$ is the degree of the covering.\footnote{This is maybe a bit surprising, as one can imagine the homology of $X_i$ consisting of large surfaces that are "expanders", and thus do not directly seem to force the 1-cycle isoperimetric constant $\rho_1(X_i;\Z)$ to be comparably small.}

The primary motivating question for this inequality is the following open problem.
\textbf{Question 1.} 
\emph{
    Is there a finitely presented residually finite group $\G$ with a nested sequence of finite index normal subgroups $\G_i\normalsub \G$ such that $\cap\G_i=\{1\}$ for which the torsion part of $\G^{ab}_i$ has cardinality that grows exponentially in the index $[\G:\G_i]$?}
    
This question appears in \cite{Nikolov} and is closely related to conjectures of Bergeron-Venkatesh, Lück, and Lê \cite{BergeronVenkatesh, Lueck, Le }, which posit the existence of such sequences in various settings.  

Theorem \ref{thm:intro tau} suggests an approach to exhibit such examples:
\begin{cor} Let $X$ be a finite 2-complex whose fundamental group $\G$ has property $(\tau)$. Let $\G_i\normalsub\G$ be a residual chain and let $X_i=\widetilde X/\G_i$ be the corresponding covers of $X$. Then if $$[\G:\G_i]\leq C \xi_2(X_i;\Z)\rho_1(X_i;\Z),$$ for some constant $C>0$, then $$0<\liminf_{i\to\infty} \frac{\log \#\G_i^{ab}}{[\G:\G_i]}.$$
\end{cor}

Note that Proposition 2.11 in \cite{RuddT} implies the constant $\rho_1(X_i;\Z)$ must tend to zero in any such sequence; thus in order to exhibit exponential torsion growth, one needs strong growth of $\xi_2(X_i;\Z)$. Note also that in the setting of hyperbolic 3-manifolds, examples exist for which an analogue of $\xi_2(X_i;\Z)$  grows exponentially in volume, and such that an analogue of $\rho_1(X_i;\Z)$ decays (nearly) exponentially in volume \cite{BrockDunfield, Rudd}.

A key step is the following geometric estimate (precise definitions are given in the next section), which is applicable in more general settings, see for instance Theorem \ref{thm:main}.
\begin{thm}\label{thm:intro-filling-contraction-and-diameter-bound}
Let $\phi:Y\to X$ be a polygonal covering map where $X$ and $Y$ are finite 2-complexes such that $\phi$ is filling and $H_2(X;\Z)\neq0$.  Then
$$
\xi_2(X;\Z) \leq \kappa_2(X;\Z)\frac{2\diam(Y^{(1)})+1}{\rho_1(X;\Z)}+1.
$$
\end{thm}
To control torsion in homology via Theorem \ref{thm:intro-filling-contraction-and-diameter-bound}, one takes $Y$ to be the cover of $X$ corresponding to the abelianization map $\pi_1X\to H_1(X;\Z)$, which is finite when $b_1(X)=0$, and then relates $\diam(Y^{(1)})$ to $\#H_1(X;\Z)$.

 This estimate is inspired by recent work of Zung in the setting of Riemannian 3-manifolds\cite{Zung}. Zung's work builds on an argument of Gromov that has been applied in myriad ways to study higher expansion and waist inequalities (see for instance \cite{Gromov, DKW,KozlovMeshulam,BaderSauer}). The most interesting new phenomenon in our setting is that the lower bound involves the largest successive minimum $\xi_2(X;\Z),$ which as noted above can potentially grow much faster than index. The analogous result in \cite{Zung} (Theorem 1.3) has a linear volume term on the left;  Proposition 2.11 in \cite{RuddT} therefore obstructs such an estimate from ever proving exponential torsion growth along residual chains.



Congruence subgroups of uniform arithmetic lattices in $\PSL_2\C$ are conjectured to have exponential torsion growth in their homology \cite{BergeronVenkatesh}.
With this 3-dimensional case in mind, in Section \ref{sec:Heegaard}, we consider 2-complexes associated to Heegaard splittings of rational homology 3-spheres. For these 2-complexes, $\kappa_2$ is uniformly bounded and one can replace the integral filling constant $\rho_1(X;\Z)$ with a rational variant $\rho_1(X;\Q)$ which should be much easier to estimate in practice. 

Another interesting class of groups to consider would be 2-dimensional residually finite hyperbolic groups with property (T); property (T) implies strong expansion properties \cite{BaderSauer}, and negative curvature provides many additional tools for studying the geometry of families of complexes.  Closely related to this class of groups are random groups with density between 1/3 and 1/2, which are 2-dimensional, hyperbolic, and have property (T), but are \emph{not} known to be residually finite, however. In any case, investigating the amount of torsion in the abelianization of finite index normal subgroups of such groups would be very interesting.

\section{Filling maps and partial contractions}\label{sec:filling}
\subsection{Definitions and notation}
In this paper we work with polygonal cellular 2-complexes. A 2-cell is a regular polygon with every boundary edge of length $1$. Note this makes every 2-complex a metric space and the path metric on the 1-skeleton is uniformly comparable to the induced metric.
A continuous map of cell complexes is polygonal if it sends $i$-cells to $i$-cells homeomorphically.
 We denote by $X^i$ the set of $i$-cells in a cell complex $X$ and by $X^{(i)}$ the $i$-skeleton.

Fix a finite 2-complex $X$. Let $(C_p(X;\Q),\d_p)$ be the rational chain complex of the cell complex $X$. To make this chain complex geometric, we equip it with the $\ell^1$ norm induced by the basis of cells. That is, if $$\sigma = \sum a_i\sigma_i\in C_p(X;\Q),$$ where $\sigma_i$ are the $p$-cells of $X$, then $$\norm{\sigma}=\sum|a_i|.$$

Let $R<\Q$ be a subring; typically $R\in\{\Z,\Q\}.$ The chain complex $C_*(X;R)$ inherits this norm.
Denote the space of $p$-cycles with coefficients in $R$ by
$$
Z_p(X;R)= \ker \big(\d_p:C_p(X;R)\to C_{p-1}(X;R)\big).
$$
A cycle $z\in Z_p(X;R)$ is $R$-exact if there is a $(p+1)$-chain $A$ with $R$-coefficients such that $\d A = z$. We will measure the complexity of exact cycles using filling functions associated to the norm $||\cdot||$ and the ring $R$. The filling function measures the minimal norm of a $(p+1)$-chain with $R$-coefficients ``filling'' $z$:
$$
\fill_p(z;R)= \inf \big\{||A||~:~ \d A=z \text{ and } A\in C_{p+1}(X;R)\big\}.
$$
Note that when $R\in\{\Z,\Q\}$, the above is a minimum. For $\Z$-coefficients, this is just a finiteness statement. With $\Q$-coefficients, this is because finding optimal fillings is a linear programming problem.

We also note that the choice of coefficient ring can have a massive effect on the filling function. For instance, take the complex corresponding to the presentation $\langle x\mid x^n,x\rangle$. Then the unique primitive $1$-cycle has an optimal integral filling with norm $1$ but over $\Q$ is filled by a chain with norm $1/n$.


Associated to a filling function is an isoperimetric constant that compares the size of a cycle to its filling. We define
$$
\rho_p(X;R)=\inf_{\substack{z\in \d C_{p+1}(X;R)\\ z\neq0}} \frac{||z||}{\fill_p(z;R)}.
$$

One can sometimes prove $\rho_p(X;\Q)=\rho_p(X;\Z)$. The rational filling constant is much more accessible, and potentially much larger, than the integral one, but it is the integral constant that is needed for the arguments in this paper. Such a coefficient upgrade is a key tool in \cite{BaderSauer} and is studied as well in \cite{Szymanski}. 
The following such coefficient upgrade follows from Theorem 2.8 in \cite{BaderSauer}, but we include a proof of the special case here as it is quite short.

Say a 2-complex is \emph{fundamental} if $H_2(X;\Z)$ is infinite cyclic and every cell in the support of the generator appears with multiplicity one. We interpret this condition as saying the 2-complex has second homology akin to a manifold's fundamental class; hence the name.

\begin{prop}\label{lem:median}
    If $X$ is a fundamental 2-complex, then  $\rho_1(X;\Z)= \rho_1(X;\Q)$.
\end{prop}
\begin{proof}
    Let $A$ be the optimal integral filling of a 1-cycle $z.$ Let $B$ be any other rational filling of $z$. Then $S=A-B$ is a 2-cycle. Because $X$ is fundamental, $S=t_BF$  for some $t_B\in\Q$, where $F=\sum \e_i\sigma_i$ is the generator of $H_2(X;\Z)$ and $\e_i\in\{0,-1,1\}.$  Since $B=A-t_BF$, finding the optimal rational filling of $z$ is equivalent to minimizing $f(t) = ||B_t||,$ where  $B_t:= A-tF$; if we can show that the minimum is realized by $t\in\Z$, then we are done.
    Write $A = \sum a_i\cdot \sigma_i,$ where $a_i\in\Z.$
    Then $$||B_t|| = \sum_{\e_i = 0}|a_{i}| + \sum_{\e_i=1}|a_i-t|+\sum_{\e_i=-1}|a_i+t|.$$
    The first summand has no effect on the minimization problem involving $t$, so we can minimize instead $$ f_0(t) = \sum_{\e_i=1}|a_i-t|+\sum_{\e_i=-1}|a_i+t|.$$
    
    The number $t$ minimizing $ f_0(t)$ is just the median of a finite set of integers given by $\e_i a_i$. We can therefore take the median to be an integer. It follows that the optimal rational filling can be taken to be integral and thus $\rho_1(X;\Z)= \rho_1(X;\Q)$.
\end{proof}

Because $X$ is a 2-complex, we can identify $H_2(X;R)$ with the space $Z_2(X;R)=\ker(\d_2)$ of all 2-cycles. We define the 2-cycle complexity $\xi_2(X;R)$  as follows:
$$\xi_2(X;R)=\inf\left\{\max_{j=1}^b||S_j||~:~ S_j \text{ generate rank $b$ submodule }H_2(X;R)\right\},$$ assuming $H_2(X;R)$ has dimension $b>0$, otherwise set $\xi_2(X;R)=0$. Note that when $R=\Q$, one always has $\xi_2(X;\Q)=0,$ so this is only interesting when the ring $R$ is small. In this note, we exclusively use $R=\Z$. In this case, one can identify $\xi_2(X;\Z)$ with the largest successive minimum of the lattice $H_2(X;\Z)\subset H_2(X;\R).$

The \emph{2-cycle curvature} $\kappa_2(X;R)$ is defined to be the maximal length (that is, its length as a cellular path) of the boundary of a 2-cell that is in the support of a 2-cycle with $R$-coefficients, assuming one exists: $$\kappa_2(X;R) = \max\left\{ \len(\d \sigma)~:~ \text{$\sigma$ is in the support of a 2-cycle in $Z_2(X;R)$}\right\}. $$ Otherwise, define $\kappa_2(X;R)=\kappa(X)$ to be the maximal length of the boundary of any 2-cell.

\subsection{Filling maps and diameter}
Let $X,Y$ be a pair of polygonal 2-complexes. Let $C_i = C_i(X;\Z)$ and $D_i = C_i(Y;\Z)$ be the corresponding cellular chain complexes with integer coefficients. 

A polygonal map $\phi:Y\to X$ induces maps $D_i\to C_i$ and $H_i(Y;\Z)\to H_i(X;\Z)$. The map $\phi$ is called \emph{filling} if $\phi_*:H_1(Y;\Z)\to H_1(X;\Z)$ is the zero map; equivalently, every $1$-cycle $z\in Z_1(Y;\Z)$ maps to an exact $1$-cycle: $\phi_*(z)\in \d C_2(X;\Z)$.

To such a filling map, we define what we call a \emph{tight partial contraction}; note that these are not unique. This will almost be a chain homotopy between $\phi$ and a constant map, but we only require it satisfy the properties of a chain homotopy in low degrees. The ``tightness'' comes from a metric condition we impose on the maps.

Recall that a chain homotopy $h_*$ between cellular maps $\phi$ and $f$ is the data of maps
$$
h_i:D_i\to C_{i+1} \text{ such that } \phi_i-f_i = \d^X h_i + h_{i-1}\d^Y.
$$

Tight partial contractions are constructed as follows.
\begin{enumerate}

    \item Fix a basepoint $y_0$ in $Y$ and set $x_0 = \phi(y_0)$. Let $f_{x_0}:Y\to X$ be the map sending every point in $Y$ to $x_0$.
    \item Fix a spanning tree $\cT_Y$ in the 1-skeleton of $Y$ with root $y_0$ such that for every $y\in Y^0$, $d_{\cT_Y}(y_0,y)\leq\diam(Y^{(1)}).$
    \item For each vertex $y\in Y,$ denote by $c_y$ the unique path in $\cT_Y$ from $y_0$ to $y.$
    \item For every vertex $y$ of $Y$, define $h_0(y) = \phi_1(c_y)$; this is the image of the chain $c_y$ corresponding to the unique path in $\cT_Y$ from $y_0$ to $y$.
    \item For an edge $e=(u,v)$ in $Y$, set
    $$
    z_e =\phi_1(e)-h_0(\d e)=\phi_1(e -c_{v}+c_{u}),
    $$
    
    which is a $1$-cycle in $X$ since $\d^Y(e-c_v+c_u) = (v-u) - (v-y_0) +(u-y_0)= 0$. 
    \item Because $\phi$ is filling, $z_e$ is exact in $X$. Let $A_e\in C_2(X;\Z)$ be a norm minimizing 2-chain satisfying $\d A_e = z_e$ and define $$h_1(e) = A_e\in C_2(X;\Z).$$
    \item Lastly, set $h_2(A) = 0\in C_3(X;\Z)=0$.
\end{enumerate}

\begin{prop}\label{prop:part-contr}
A tight partial contraction satisfies the chain homotopy condition between $\phi$ and the corresponding constant map in degree $0$ and $1$.
\end{prop}

\begin{proof}
We need to verify $$\phi_i-(f_{x_0})_i = \d^X h_i + h_{i-1}\d^Y$$ for all cells of the appropriate degree. In degree zero, $h_{-1}$ is the zero map, so we can verify immediately that $\phi_0(y) - x_0 =  \d^X h_0(y)$.

In degree 1, we need to verify for each 1-cell $e$ in $Y$ that $$\phi_1(e) = \d^X h_1(e) + h_{0}(\d^Ye).$$ 
Let $e = (u,v)$.
Recall that in the construction of $h_1$, we defined $h_1(e)$ to be a chain $A_e\in C_2(X;\Z)$ satisfying $\d^X h_1(e) = \d^X A_e = \phi_1(e)-h_0(\d^Y e).$ Therefore, the right-hand side of the equation above becomes $$\d^Xh_1(e) + h_0(\d^Ye) = \phi_1(e) - h_0(\d^Ye) + h_0(\d^Ye) = \phi_1(e),$$ as desired.
\end{proof}

The following is immediate from the definition of chain homotopy and the previous proposition.

\begin{lem}
A tight partial contraction $h_i:C_i(Y;\Z)\to C_{i+1}(X;\Z)$ of a polygonal map $\phi:Y\to X$ is a chain homotopy between $\phi$ and the corresponding constant map if and only if for every 2-cell $A$ one has
$$
\phi_2(A) -h_{1}(\d^Y A)=0.
$$
\end{lem}

We are interested in what happens when the tight partial contraction does not come from a chain homotopy to the constant map. The following lemma is the main technical tool in this paper.

\begin{lem}\label{lem:2-cell-homological}
Let $\phi:Y\to X$ be a degree $d$ polygonal covering map of finite 2-complexes with $H_2(X;\Z)\neq0$. Then for any tight partial contraction $h_i:C_i(Y;\Z)\to C_{i+1}(X;\Z)$ for $\phi$, there exist 2-cells $\tilde \sigma_1,\dots,\tilde \sigma_b$ that project to cells in the support of 2-cycles in $X$, such that $$\{\phi_2(\tilde\sigma_i)-h_1(\d^Y\tilde\sigma_i)\}_1^b$$ span a full rank subgroup of $H_2(X;\Z)$.
\end{lem}

\begin{proof}

     Consider a nonzero 2-cycle $S=\sum_{\sigma\in X^{2}} a_\sigma \sigma$ in $H_2(X;\Z)$ and define $\tilde S=\sum_{\tilde\sigma\in Y^2} a_{\tilde \sigma}\tilde\sigma,$ where $a_{\tilde\sigma}=a_{\phi_2(\tilde\sigma)}$; this is the image of $S$ under the transfer map. 
     The transfer map takes cycles to cycles, so $\tilde S$ is a 2-cycle, and thus $$h_1(\d^Y \tilde S)=h_1(0)=0.$$

Denote by $\fc:C_2(Y;\Z)\to C_2(X;\Z)$ the map induced by $$\fc(\tilde\sigma) =\phi_2(\tilde \sigma) -h_{1}(\d^Y\tilde \sigma).$$ First note that the partial chain homotopy ensures this is always a cycle:
\begin{align*}
    \d^X \fc(\tilde\sigma) &=\d^X \phi_2(\tilde \sigma) -\d^Xh_{1}(\d^Y\tilde \sigma)\\
    &=\d^X \phi_2(\tilde \sigma) - (\phi_1(\d^Y \tilde \sigma)-h_0(\d^Y\d^Y\tilde\sigma)) \text{, by Proposition \ref{prop:part-contr}}\\
    &=\d^X \phi_2(\tilde \sigma) - \phi_1(\d^Y \tilde \sigma)\\
    &=0.
\end{align*}
Next, observe that $$\fc(\tilde S) = \phi_2(\tilde S)-h_1(\d^Y \tilde S)  = \phi_2(\tilde S)=dS.$$
Now consider the span of the set $$\cB=\{\fc(\tilde\sigma)~:~\tilde\sigma \text{ lies in the support of a cycle in the image of the transfer map}\}.$$ Using our previous observation, we find that for every 2-cycle $S\in H_2(X;\Z),$ the cycle $dS$ is a linear combination of elements of $\cB$. It follows that there exist 2-cells $\tilde\sigma_1,\dots\tilde\sigma_b$ such that $\{\fc(\tilde\sigma_i)\}$ generate a full rank subgroup of $H_2(X;\Z),$ and these 2-cells can be taken to project to cells in the support of 2-cycles in $X$, as claimed.
\end{proof}

We next record some quantitative properties of tight partial contractions associated to a polygonal map.

\begin{prop}\label{prop:quant-partial-contraction}
Let $\phi:Y\to X$ be a polygonal map where $X$ and $Y$ are 2-complexes. Let $h_i$ be a tight partial contraction over $\Z$.
\begin{enumerate}
    \item For any vertex $u$ in $Y$, there is a bound $||h_0(u)||\leq \diam(Y^{(1)})$.
    \item For any edge $e$ in $Y$, there is a bound
    $$
    ||h_1(e)|| \leq \frac{2\diam(Y^{(1)})+1}{\rho_1(X;\Z)}.
    $$
\end{enumerate}
\end{prop}

\begin{proof}
    The first estimate is immediate from the definition of $h_0$ using a tree and the fact the map does not increase the length of paths. 
    The second estimate follows from the filling construction of $h_1$. First note that if $z_e$ is trivial, $h_1(e)=0$, and the inequality is trivial. We therefore can assume $z_e\neq0.$ Then, by construction, $||h_1(e)|| = \fill_1(z_e,\Z)$. Hence, 
$$    \rho_1(X;\Z)\leq  \frac{||z_e||}{||h_1(e)||}$$
The definition of $z_e$ using the spanning tree in $Y$ and the fact the map is polygonal, so does not increase lengths, immediately gives $$\norm{z_e}\leq 2\diam(Y^{(1)})+1.$$
Therefore $$ \rho_1(X;\Z)\leq \frac{2\diam(Y^{(1)})+1}{\norm{h_1(e)}}$$ and rearranging gives the claim.
\end{proof}

We can now prove our main geometric estimate.

\begin{thm}\label{thm:filling-contraction-and-diameter-bound}
Let $\phi:Y\to X$ be a polygonal covering map where $X$ and $Y$ are finite 2-complexes, $\phi$ is filling over $\Z$, and assume $H_2(X;\Z)\neq0$. Let $h_i$ be a tight partial contraction for $\phi$ over $\Z$. Then 
$$
\xi_2(X;\Z) \leq \kappa_2(X;\Z)\frac{2\diam(Y^{(1)})+1}{\rho_1(X;\Z)}+1.
$$
\end{thm}
\begin{proof} Let $\cA=\{\tilde\sigma_1,\dots,\tilde\sigma_b\}$ be the set of 2-cells given by Lemma \ref{lem:2-cell-homological}. Let $A \in\cA $ be such that $||h_1(\d^Y A) - \phi_2(A) ||$ is maximized over all 2-cells in $\cA;$ set $S= h_1(\d^Y A) - \phi_2(A)$. It follows from the fact the cycles $\{h_1(\d^Y\tilde\sigma_i)-\phi_2(\tilde\sigma_i)\}_{1}^b$ spans a full rank subset of $H_2(X;\Z)$, that $$\xi_2(X;\Z) \leq ||S||.$$  
Write $\d A = \sum a_ie_i$, so that $h_1(\d A) = \sum a_ih_1(e_i)$.
By Proposition \ref{prop:quant-partial-contraction}, we can estimate $$||h_1(\d A) ||\leq \sum ||a_ih_1(e_i)||\leq \sum |a_i|\frac{2\diam(Y^{(1)})+1}{\rho_1(X;\Z)}\leq\kappa_2(X;\Z)\frac{2\diam(Y^{(1)})+1}{\rho_1(X;\Z)}.$$ 
Therefore,
$$
||S||  = ||h_1(\d A) - \phi_2(A) || \leq ||h_1(\d A)||+||A|| \leq \kappa_2(X;\Z)\frac{2\diam(Y^{(1)})+1}{ \rho_1(X;\Z)}+1,$$
 where we applied Proposition \ref{prop:quant-partial-contraction} to each summand of $\d A$ in the final step. 
\end{proof}

Torsion bounds will come from estimating diameters of universal abelian covers in the next section. The following proposition puts universal abelian covers into the context of filling maps and partial contractions.

\begin{prop}\label{prop:uni-cover-filling}
Let $X$ be a finite 2-complex with trivial first Betti number.
Let $\phi:Y\to X$ be the universal abelian cover; then $\phi$ is filling over $\Z$ so admits a tight partial contraction. 
\end{prop}
\begin{proof}
    Because $X$ has trivial first Betti number, $Y$ is finite. Because $Y$ is the universal abelian cover, any homologically nontrivial cycle $z$ in $Z_1(X;\Z)$ defines a nontrivial deck transformation and lifts to a non-closed curve. Hence, the image of any 1-cycle in $Y$ into $X$ must be nullhomologous, which implies the map $\phi_*:H_1(Y;\Z)\to H_1(X;\Z)$ is zero and therefore $\phi$ is filling.
 \end{proof}

\section{Diameter bounds and torsion}

We obtain bounds on torsion in homology using tight partial contractions of the universal abelian covering map. The key is to bound the diameter term in Theorem \ref{thm:filling-contraction-and-diameter-bound} with something involving the size of the first homology (which is exactly the degree of the covering map).

An interesting class of complexes to consider are those whose fundamental groups have property $(\tau)$, as the diameters of their finite index covers grow logarithmically in the covering degree and all finite index normal subgroups are rationally perfect. See \cite{LZ} for a general reference on property $(\tau).$

The following is standard; see for instance Proposition 5.24 in \cite{LZ}.
\begin{lem}\label{lem:tau}
    Let $X$ be a 2-complex with fundamental group $\G$ with property $(\tau)$. Let $Y\to X$ be a regular finite cover of degree $d$. Then $$\diam(Y^{(1)}) \leq  C(1+\log d)$$ for a constant $C$ depending only on $X.$
\end{lem}

With this, we can prove our main theorem.

\begin{thm}\label{thm: tau}
    Let $X_0$ be a finite 2-complex whose fundamental group $\G_0$ has property $(\tau)$. Let $\G\normalsub \G_0$ be a finite index normal subgroup. Let $X =\widetilde X_0/\G$. Then there is a constant $C$ depending only on $X_0$ such that  $$
(\xi_2(X;\Z)-1)\rho_1(X;\Z)-C(1+\log([\G_0:\G])) \leq  C\log\#\G^{ab}.
$$ In particular, for a residual tower of covers $X_i\to X_0,$ we have $$\xi_2(X_i;\Z)\rho_1(X_i;\Z)-o([\G_0:\G_i])\leq C\log\#\G^{ab}_i.$$
\end{thm}

\begin{proof}
     Because the composition $\widetilde X^{ab}\to X\to X_0$ is regular of degree $[\G_0:\G]\#H_1(X;\Z)$, we have by
          Lemma \ref{lem:tau} that $$\diam(\widetilde X^{ab})\leq C(1+\log ([\G_0:\G]\#H_1(X;\Z))),$$ where $C$ depends only on $X_0.$

 By Theorem \ref{thm:filling-contraction-and-diameter-bound} and Proposition \ref{prop:uni-cover-filling}, we can use this estimate to obtain
$$
\xi_2(X;\Z) \leq \kappa_2(X)\frac{2 C(1+\log ([\G_0:\G]\#H_1(X;\Z)))+1}{\rho_1(X;\Z)}+1.
$$
Now rearrange and note $\kappa_2(X)$ is uniformly bounded over all finite covers to obtain for some larger constant $C$ depending only on $X_0$
 $$
(\xi_2(X;\Z)-1)\rho_1(X;\Z) \leq C(1+\log ([\G_0:\G])) + C\log\#H_1(X;\Z).
$$

This immediately implies the first statement in the theorem. The second statement follows from the first combined with Proposition 2.11 in \cite{RuddT}, which says that if $X_i\to X_0$ are a residual tower of covers and $b_1(X_i)=0$ for all $i,$ then $\rho_1(X_i;\Z)\to 0$.
\end{proof}

Zung adapts the main result of Benjamini-Finucane-Tessera in \cite{BFT} to prove the following related diameter bound, which he uses for estimates analogous to those in Theorem \ref{thm:main} for Riemannian 3-manifolds.

\begin{lem}[\cite{Zung} Lemma 6.4]\label{lem:Zung}
For any $\delta>0$, there is a constant $C_\delta$ such that if $X$ is a metric space with $\#H_1(X)<\infty$ and $\widetilde X^{ab}$ is the cover of $X$ corresponding to the kernel of the map $ab:\pi_1X\to H_1(X)$, then
$$
\diam(\widetilde X^{ab}) \leq C_\delta \#H_1(X)^{\delta}\diam(X).
$$
\end{lem}

Note that the diameter used by Zung is the diameter of the 2-complex, not the 1-skeleton, but as this is comparable to the diameter of the 1-skeleton, up to universal constants, the distinction does not matter here. This then gives the following theorem; as the proof is essentially identical to that of Theorem \ref{thm: tau}, we omit it.

\begin{thm}\label{thm:main}
Let $X$ be a 2-complex with rationally perfect fundamental group $\G$. Then for any $\delta>0$, there is a constant $C_\delta>0$ depending only on $\delta$ such that
$$
\frac{\xi_2(X;\Z)-1}{\kappa_2(X;\Z)\diam(X)}\rho_1(X;\Z)\leq C_\delta(\#\G^{ab})^{\delta}.
$$
\end{thm}



\section{Heegaard diagrams and associated 2-complexes}\label{sec:Heegaard}

In this section we discuss certain 2-complexes associated to Heegaard splittings and list some properties motivated by the previous sections.

A \emph{Heegaard diagram} describes a 3-manifold $M$ as a triple $(\S,\ua,\ub)$, where $\S$ is a genus $g$ surface and $\ua = \cup_i\alpha_i$ and $\ub = \cup_i\beta_i$ are embedded multicurves with $g$ components such that $\S-\ua$ and $\S-\ub$ are connected.
From a diagram, one can associate a 3-manifold with boundary $S^2\sqcup S^2$ by thickening the surface $\S$, then attaching thickened disks along the components of $\ua$ and $\ub$. Filling in the two $S^2$'s by attaching 3-balls results in a closed 3-manifold $M$. We denote the closed 3-manifold built from this procedure by $M(\S,\ua,\ub)$.

A \emph{combinatorial Heegaard diagram} is a triple $(\S,\ta,\tb)$ where $\S$ is a triangulated surface and $\ta$ and $\tb$ are polygonal $g$-component multicurves such that $\ta$ and $\tb$ are homotopic to multicurves $\ua$ and $\ub$ with $\S-\ua$ and $\S-\ub$ connected.

We denote by $||\S||$ the number of 2-cells in the underlying triangulation of $\S$. Of course, every combinatorial Heegaard diagram comes from a Heegaard diagram, so we can as above associate to a combinatorial Heegaard diagram a closed 3-manifold $M(\S,\ta,\tb)$.

Let $X(\S,\ta,\tb)$ be the 2-complex obtained by attaching polygonal disks to $\S$ along each polygonal loop in $\ta$ and $\tb$. This 2-complex is homotopy equivalent to $M(\S,\ta,\tb)-(B^3_0\cup B^3_1)$, where $B^3_0$ and $B^3_1$ are any pair of embedded disjoint 3-balls.

\begin{prop}\label{prop:heeg}
The 2-complex $X = X(\S,\ta,\tb)$ associated to a combinatorial Heegaard diagram satisfies $\pi_1(X) \cong \pi_1(M(\S,\ta,\tb))$. When $b_1(X) = 0$, the 2-complex $X$ furthermore satisfies the following:
\begin{enumerate}
    \item $b_2(X)=1$.
    \item The 2-cycle complexity $\xi_2(X;\Z)$ is given by the number of 2-cells in the triangulation of $\S$:
    $$
    \xi_2(X;\Z) = ||\S||.
    $$
    \item $X$ is a fundamental 2-complex.
    \item The 2-cycle curvature of $X$ satisfies $\kappa_2(X;\Z) = 3$.
\end{enumerate}
\end{prop}
\begin{proof}
Deleting a pair of 3-balls from $M$ does not change the fundamental group and by construction $X$ is homotopy equivalent to $M$ with a pair of 3-balls removed. When $b_1(X)=0$, the second homology is generated by one of the boundary components. The 2-chain in $C_2(X;\Z)$ corresponding to the triangulation of $\S$ is a primitive 2-cycle, so because we are in a 2-complex with second Betti number one it must be the generator. It follows immediately that $\xi_2(X;\Z) = ||\S||$ and as every 2-cell in $\S$ appears exactly once, the 2-complex is fundamental. The 2-cycle curvature is realized by any triangular face in $\S$, giving $\kappa_2(X;\Z) = 3$.
\end{proof}

Combining the above construction with our previous estimates gives the following interesting bound.

    \begin{prop}     Let $M$ be a closed 3-manifold with trivial first Betti number. Fix a combinatorial Heegaard splitting $(\S,\ta,\tb)$ of $M$ and let $X$ be the associated 2-complex. 

    Then for every $\delta>0$, there is a constant $C_\delta$ depending only on the comparison between triangulation complexity of $\S$ and the genus of $\S$ such that $$\frac{\genus(\S)\cdot\rho_1(X;\Q)}{\diam(X)}\leq C_\delta\#H_1(M;\Z)^{\delta}.$$
        
    \end{prop}

    \begin{proof}
     Combine Proposition \ref{prop:heeg} with Theorem \ref{thm:main} and Proposition \ref{lem:median} and combine constants to obtain the proposition.
    \end{proof}

    Recall that Lackenby showed there exist families of covers $M_i\to M$ such that for any Heegaard surface $\S_i$ for $M_i$, one has $\genus(\S_i)/\vol(M_i)>\e$ \cite{LackenbySplit}. However, finding covers admitting diagrams for which one can show $\rho_1(X_i)/\diam(X_i)$ decays sufficiently slowly seems difficult.

\bibliographystyle{alpha}
\footnotesize{
\bibliography{bib}}
\end{document}